\documentclass[11pt,a4paper,reqno]{amsart}
\usepackage{hyperref}
\usepackage[all]{xy}

\usepackage{amssymb,amsmath,latexsym,amsthm}
\usepackage[english]{babel}

\usepackage[all]{xy}
\usepackage{mathtools}

\usepackage{color}
\usepackage[usernames,dvipsnames]{xcolor}

\usepackage{tikz}

\newcommand{\xrightarrowdbl}[2][]{%
  \xrightarrow[#1]{#2}\mathrel{\mkern-14mu}\rightarrow
}

\newcommand{\fp}{\mathfrak{p}}

\newcommand{\Hom}{\text{Hom}}
\newcommand{\End}{\text{End}}

\newcommand{\op}{\text{op}}

\newcommand{\ad}{\text{ad}}

\newcommand{\Sp}{\text{Spec} \,}
\newcommand{\Lie}{\text{Lie}}

\chardef\bslash=`\\ 

\newtheorem{theorem}{Theorem}[section]

\newtheorem{prop}[theorem]{Proposition}
\newtheorem{lem}[theorem]{Lemma}
\newtheorem{cor}[theorem]{Corollary}

\theoremstyle{definition}

\newtheorem{remark}[theorem]{Remark}
\newtheorem{example}[theorem]{Example}

\numberwithin{equation}{section}

\newtheorem*{maintheorem*}{Main Theorem}
\theoremstyle{definition}
\newtheorem{definition}{Definition}

\newcommand{\surj}{\twoheadrightarrow}

\newcommand{\CO}{\mathcal{O}}
\newcommand{\CS}{\mathcal{O}_S}

\newcommand{\Oiy}{\mathcal{O}_{\{\iy\}}}
\newcommand{\ClS}{\text{Cl}_S}

\newcommand{\fX}{\mathfrak{X}}

\newcommand{\mk}{\medskip}
\renewcommand{\sectionmark}[1]{}
\renewcommand{\Im}{\operatorname{Im}}

\newcommand{\diag}{\text{diag}}

\newcommand{\af}{\text{af}}
\newcommand{\Br}{{\mathrm{Br}}}
\newcommand{\Aut}{\textbf{Aut}}

\newcommand{\ve}{\varepsilon}
\newcommand{\iy}{\infty}

\newcommand{\bk}{\bigskip}
\newcommand{\fc}{\frac}

\newcommand{\s}{\sigma}
\newcommand{\syc}{\text{sc}}

\newcommand{\Pic}{\text{Pic~}}

\newcommand{\dl}{\delta}

\newcommand{\lm}{\lambda}
\newcommand{\Lm}{\Lambda}

\newcommand{\Om}{\Omega}
\newcommand{\ov}{\overline}
\newcommand{\un}{\underline}

\newcommand{\BG}{\mathbb{G}}
\newcommand{\BF}{\mathbb{F}}

\newcommand{\C}{\textbf{C}}

\newcommand{\BQ}{\mathbb{Q}}

\newcommand{\Z}{\mathbb{Z}}

\renewcommand{\a}{\alpha}
\renewcommand{\b}{\beta}

\newcommand{\et}{\text{\'et}}
\newcommand{\zar}{\text{Zar}}

\newcommand{\BZ}{\mathbb{Z}}
\newcommand{\A}{\mathbb{A}}
\newcommand{\N}{\mathbb{N}}

\newcommand{\Nrd}{{\mathrm{Nrd}}}
\newcommand{\TwG}{\textbf{Twist}(\un{G})}

\begin{document}

\title[Twisted forms of a semisimple group]{The twisted forms of a semisimple group over an $\BF_q$-curve}

\author{Rony A. Bitan, Ralf K\"ohl, Claudia Schoemann}
\address{Rony A. Bitan, \\ Afeka, Tel-Aviv Academic College of Engineering \\ Tel-Aviv, Israel, \\ Bar-Ilan University \\ Ramat-Gan, Israel}
\email{ronyb@afeka.ac.il}

\address{Ralf K\"ohl, \\ JLU Giessen, Mathematisches Institut, Arndtstr.~2, 35392 Giessen, Germany}
\email{ralf.koehl@math.uni-giessen.de}

\address{Claudia Schoemann, \\ Leibniz University Hannover, Institute for Algebraic Geometry, Welfengarten 1, 30167 Hannover, Germany}
\email{schoemann@math.uni-hannover.de}

\subjclass[2010]{11G20,11G45,11R29}

\keywords{Class number, Hasse principle, Tamagawa number, \'etale cohomology}

\maketitle


\begin{abstract}
Let $C$ be a smooth, projective and geometrically connected curve defined over a finite field $\BF_q$.
Given a semisimple $C-S$-group scheme $\un{G}$ where $S$ is a finite set of closed points of $C$,
we describe the set of ($\CS$-classes of) twisted forms of $\un{G}$
in terms of geometric invariants of its fundamental group $F(\un{G})$.
\end{abstract}

\bk

\section{Introduction} \label{Introduction}
Let $C$ be a projective, smooth and geometrically connected curve defined over a finite field $\BF_q$.
Let $\Om$ be the set of all closed points on $C$.
For any $\fp \in \Om$ let $v_\fp$ be the induced discrete valuation
on the (global) function field $K=\BF_q(C)$,
$\hat{\CO}_\fp$ the ring of integers in the completion $\hat{K}_\fp$ of $K$ with respect to $v_\fp$,
and $k_\fp$ the residue field.
Any finite subset $S \subset \Om$ gives rise to a \emph{Dedekind scheme},
namely, a Noetherian integral scheme of dimension $1$ whose local rings are regular;
If $S$ is \emph{nonempty} it will be the spectrum of the Dedekind domain
$$ \CS :=
\{x \in K: v_\fp(x) \geq 0 \ \forall \fp \notin S \}.  $$
Otherwise, if $S = \emptyset$, the corresponding Dedekind scheme is the curve $C$ itself, 
and we denote by $\CS$ the structural sheaf of $C$. 

\mk

Throughout this paper $\un{G}$ is an $\CS$-group scheme whose generic fiber  $G:=\un{G} \otimes_{\CS} K$ is almost-simple,
and whose fiber $\un{G}_\fp=\un{G} \otimes_{\CS} \hat{\CO}_\fp$ at any $\fp \in \Om-S$ is \emph{semisimple},
namely, (connected) reductive over $k_\fp$,
and the rank of its root system equals that of its lattice of weights (\cite[Exp.~XIX Def.~2.7, Exp. XXI Def.~1.1.1]{SGA3}).
Let $\un{G}^\syc$ be the universal (central) cover (being simply-connected) of $\un{G}$,
and suppose that its \emph{fundamental group} $F(\un{G}):=\ker[\un{G}^\syc \surj \un{G}]$ (cf. \cite[p.40]{Con2}) is of order prime to $\text{char}(K)$.

\mk

A \emph{twisted form} of $\un{G}$ is an $\CS$-group that is isomorphic to $\un{G}$
over some finite \'etale cover of~$\CS$.
We aim to describe explicitly --
in terms of some invariants of $F(\un{G})$ and the group of outer automorphisms of~$\un{G}$ --
the finite set of all twisted forms of $\un{G}$, modulo $\CS$-isomorphisms.
This is done first in Section~\ref{torsors classification} for forms arising from the torsors of the adjoint group $\un{G}^\ad$,
and then in Section \ref{twisted forms},
through the action of the outer automorphisms of $\un{G}$ on its Dynkin diagram,
for all twisted forms.
More concrete computations
are provided in Sections \ref{Split fundamental group}, \ref{Quasi-split fundamental group} and \ref{Non quasi-split fundamental group}.
The case of type $\text{A}$ deserves a special consideration,
this is done in Section \ref{anisotropic case}.
The Zariski topology is treated in Section \ref{Zariski}.

\mk

Before we start may we quote B. Conrad in the abstract of \cite{Con3}:
``The study of such $\Z$-groups provides concrete applications of many facets
of the theory of reductive groups over rings (scheme of Borel subgroups, auto-morphism scheme, relative non-abelian cohomology, etc.),
and it highlights the role of number theory (class field theory, mass formulas, strong approximation,
point-counting over finite fields, etc.) in analyzing the possibilities".

\mk

\section{Torsors} \label{torsors classification}
A \emph{$\un{G}$-torsor $P$ in the \'etale topology} is a sheaf of sets on $\CS$ 
equipped with a (right) $\un{G}$-action, which is locally trivial in the \'etale topology, 
namely, locally for the \'etale topology on $\CS$, this action is isomorphic to the action of $G$ on itself by translation.
The associated $\CS$-group scheme ${^P}\un{G}$, being an inner form of $\un{G}$,
is called the \emph{twist} of $\un{G}$ by $P$
(e.g., \cite[\S 2.2, Lemma 2.2.3, Examples 1,2]{Sko}).
We define $H^1_\et(\CS,\un{G})$ to be the set of isomorphism classes
of $\un{G}$-torsors relative to the \'etale topology (or the flat one;
these two cohomology sets coincide when $\un{G}$ is smooth; cf. \cite[VIII~Cor.~2.3]{SGA4}).
This set is finite (\cite[Prop.~3.9]{BP}).
The sets $H^1(K,G)$ (denoting the Galois cohomology) and $H^1_\et(\hat{\CO}_\fp,\un{G}_\fp)$ are defined similarly.

\mk

There exists a canonical map of pointed-sets:
\begin{equation*}
\lm: H^1_\et(\CS,\un{G}) \to H^1(K,G) \times \prod\limits_{\fp \notin S} H^1_\et(\hat{\CO}_\fp,\un{G}_\fp).
\end{equation*}
defined by $[X] \mapsto [(X \otimes_{\CS} \Sp K) \times \prod_{\fp \notin S} X \otimes_{\CS} \Sp \hat{\CO}_\fp]$.
Let $[\xi_0] := \lm([\un{G}])$.
The \emph{principal genus} of $\un{G}$ is then $\ker(\lm) = \lm^{-1}([\xi_0])$,
namely, the classes of $\un{G}$-torsors that are generically and locally trivial at all points of $\CS$.
More generally, a \emph{genus} of $\un{G}$ is any fiber $\lm^{-1}([\xi])$ where $[\xi] \in \Im(\lm)$.
The \emph{set of genera} of $\un{G}$ is then:
$$ \text{gen}(\un{G}) := \{ \lm^{-1}([\xi]) \ : \ [\xi] \in \Im(\lm) \}, $$
hence $H^1_\et(\CS,\un{G})$ is a disjoint union of all genera.

\mk

The ring of $S$-integral ad\`eles $\A_S := \prod_{\fp \in S} \hat{K}_\fp \times \prod_{\fp \notin S} \hat{\CO}_\fp$ is a subring of the ad\`eles $\A$.
A $\un{G}$-torsor $P=\text{Iso}(\un{G},\un{G}')$ belongs to the principal genus of $\un{G}$ 
if it is both $\A_S$- and $K$-trivial, hence the principal genus
bijects as a pointed-set to the $S$-\emph{class set} of $\un{G}$ (see \cite[Thm.~I.3.5]{Nis}):
$$ \ClS(\un{G}) := \un{G}(\A_S) \backslash \un{G}(\A) / G(K).   $$
Being finite (\cite[Proposition~3.9]{BP}), its cardinality, called the $S$-\emph{class number} of $\un{G}$, is denoted $h_S(\un{G})$.
As $\un{G}$ is assumed to have connected fibers, 
by Lang's Theorem (recall that all residue fields are finite) all $H^1_\et(\hat{\CO}_\fp,\un{G}_\fp)$ vanish,
which indicates that any two $\un{G}$-torsors share the same genus if and only if they are $K$-isomorphic.

\mk

The universal cover of $\un{G}$ forms a short exact sequence of \'etale $\CS$-groups (cf. \cite[p.40]{Con2}):
\begin{equation} \label{universal covering}
1 \to F(\un{G}) \to \un{G}^\syc \to \un{G} \to 1.
\end{equation}
This gives rise by \'etale cohomology to the co-boundary map of pointed sets:
\begin{equation} \label{delta}
\dl_{\un{G}} : H^1_\et(\CS,\un{G}) \surj H^2_\et(\CS,F(\un{G}))
\end{equation}
which is surjective by (\cite{Dou}, Cor.~1) as $\CS$ is of \emph{Douai-type}
(see Def.~5.2 and Exam.~5.4(iii),(v) in \cite{Gon}).
It follows from the fact that $ H^2_\et(\CS,\un{G}^\syc) $
(resp., $ H^2_\et(\CS,\un{G}^\syc) $) has only trivial classes
and in finite number (\cite{Dou}, Thm. 1.1).

\mk

A representation $\rho:\un{G}^\syc \to \un{\textbf{GL}}_1(A)$ where $A$ is an Azumaya $\CS$ algebra,
is said to be \emph{center-preserving} if $\rho(Z(\un{G})^\syc) \subseteq Z(\un{\textbf{GL}}_1(A))$.
The restriction of $\rho$ to $F(\un{G}) \subseteq Z(\un{G}^\syc)$,
composed with the natural isomorphism $Z(\un{\textbf{GL}}_1(A)) \cong \un{\BG}_m$,
is a map $\Lm_\rho: F(\un{G}) \to \un{\BG}_m$,
thus inducing a map: $(\Lm_\rho)_*:H^2_\et(\CS,F(\un{G})) \to H^2_\et(\CS,\un{\BG}_m) \cong \Br(\CS)$.
Together with the preceding map $\dl_{\un{G}}$ we get the map of pointed-sets:
\begin{equation} \label{relative Tits}
(\Lm_{\rho})_* \circ \dl_{\un{G}}: H^1_\et(\CS,\un{G}) \to \Br(\CS),
\end{equation}
which associates any class of $\un{G}$-torsors with a class of Azuamaya $\CS$-algebras in $\Br(\CS)$.

\bk

When $F(\un{G})=\un{\mu}_m$, the following composition is surjective:
\begin{equation} \label{w_G}
w_{\un{G}}: H^1_\et(\CS,\un{G}) \xrightarrowdbl{\dl_{\un{G}}} H^2_\et(\CS,F(\un{G})) \xrightarrowdbl{i_*^{(2)}} {_m}\Br(\CS),
\end{equation}
and coincides with $(\Lm_\rho)_* \circ \dl_{\un{G}}$.

\mk

The original \emph{Tits algebras} introduced in \cite{Tits71},
are central simple algebras defined over a field,
associated to algebraic groups defined over that field.
This construction was generalized to group-schemes over rings as shown in \cite[Thm.1]{PS}.
We briefly recall it here over $\CS$:
Being semisimple, $\un{G}$ admits an inner form $\un{G}_0$ which is \emph{quasi-split}
(in the sense of \cite[XXIV, 3.9]{SGA3},
namely, not only requiring a Borel subgroup to be defined over $C-S$
but some additional data involving the scheme of Dynkin diagrams, see \cite[Def.~5.2.10.]{Con2}). 

\begin{definition} \label{Tits algebra}
Any center-preserving representation $\rho_0:\un{G}_0 \to \un{\textbf{GL}}(V)$
gives rise to a ``twisted" center-preserving representation:
$\rho:\un{G} \to \un{\textbf{GL}}_1(A_\rho)$, where $A_\rho$ is an Azumaya $\CS$-algebra,
called the \emph{Tits algebra corresponding to the representation $\rho$},
and its class in $\Br(\CS)$, is its \emph{Tits class}.
\end{definition}

\begin{lem} \label{Tits class}
If $\un{G}$ is adjoint, then for any center-preserving representation $\rho$ of $\un{G}^\syc_0$,
and a twisted $\un{G}$-form ${^P}\un{G}$ by a $\un{G}$-torsor $P$, one has:
$((\Lm_{\rho})_* \circ \dl_{\un{G}})([{^P}\un{G}])] = [{^P}A_\rho] - [A_\rho] \in \Br(\CS)$
where $[{^P}A_\rho]$ and $[A_\rho]$ are the Tits classes of $({^P}\un{G})^\syc$ and $\un{G}^\syc$ corresponding to $\rho$, respectively.
\end{lem}

\begin{proof}
By descent $F(\un{G}_0) \cong F(\un{G})$,
so we may write the short exact sequences of $\CS$-groups:
\begin{align}
1 \to F(\un{G}) \to &\un{G}^\syc     \to \un{G} \to 1  \\ \notag
1 \to F(\un{G}) \to &\un{G}^\syc_0 \to \un{G}_0 \to 1
\end{align}
which yield the following commutative diagram of pointed sets (cf. \cite[IV, Prop.~4.3.4]{Gir}):
\begin{equation} \label{relative w}
\xymatrix{
H^1_\et(\CS,\un{G}_0) \ar[r]^{=} \ar[d]^{\dl_0}        &  H^1_\et(\CS,\un{G}) \ar[d]^{\dl_{\un{G}}}  \\
H^2_\et(\CS,F(\un{G}))                \ar[r]^{r_{\un{G}}}_{=} &  H^2_\et(\CS,F(\un{G}))
}
\end{equation}
in which $r_{\un{G}}(x) := x - \dl_0([\un{G}])$,
so that $\dl_{\un{G}} = r_{\un{G}} \circ \dl_0$ maps $[\un{G}]$ to $[0]$.
The image of any twisted form ${^P}\un{G}$ where $[P] \in H^1_\et(\CS,\un{G})$ (see in Section \ref{Introduction}),
under the coboundary map
$$\dl:H^1_\et(\CS,\un{G}_0) \to H^2_\et(\CS,Z(\un{G}_0^\syc))$$
induced by the universal covering of $\un{G}_0$ corresponding to $\rho$, is $[{^P}A_\rho]$,
where ${^P}A_\rho$ is the Tits-algebra of $({^P}\un{G})^\syc$ (see \cite[Theorem~1]{PS}).
But $\un{G}_0$ is adjoint, so $Z(\un{G}_0^\syc) = F(\un{G}_0) \cong F(\un{G})$,
thus the images of $\dl$ and $\dl_0$ coincide in $\Br(\CS)$, whence:
\begin{equation*}
((\Lm_{\rho})_*(\dl_{\un{G}}([\un{G}'])) = ((\Lm_{\rho})_*(\dl_0([{^P}\un{G}]) - \dl_0([\un{G}])) = [{^P}A_\rho]-[A_\rho]. \qedhere
\end{equation*} 
\end{proof}

The fundamental group $F(\un{G})$ is a finite, of multiplicative type (cf. \cite[XXII, Cor.~4.1.7]{SGA3}),
commutative and smooth $\CS$-group (as its order is assumed prime to $\text{char}(K)$).

\begin{lem} \label{H1G iso to H^2F}
If $\un{G}$ is not of type $\text{A}$, or $S=\emptyset$,
then $H^1_\et(\CS,\un{G})$ is isomorphic to $H^2_\et(\CS,F(\un{G}))$.
\end{lem}

\begin{proof}
Applying \'etale cohomology to sequence \eqref{universal covering} yields the exact sequence:
$$ H^1_\et(\CS,\un{G}^\syc) \to H^1_\et(\CS,\un{G}) \xrightarrow{\delta_{\un{G}}} H^2_\et(\CS,F(\un{G}))	 $$
in which $\delta_{\un{G}}$ is surjective (see \eqref{delta}).
If $\un{G}$ is not of absolute type $\text{A}$,
it is locally isotropic everywhere (\cite[4.3~and~4.4]{BT}), in particular at $S$.
This is of course redundant when $S=\emptyset$.
Thus $H^1_\et(\CS,\un{G}^\syc)$ vanishes (\cite[Lemma 2.3]{Bit2}).
Changing the base-point in $H^1_\et(\CS,\un{G})$ to any $\un{G}$-torsor $P$,
it is bijective to $H^1_\et(\CS,{^P}\un{G})$ where ${^P}\un{G}$ is an inner form of $\un{G}$
(see Section \ref{Introduction}), thus an $\CS$-group of the same type.
Similarly all fibers of $\dl_{\un{G}}$ vanish. 
This amounts to $\delta_{\un{G}}$ being injective thus an~isomorphism. 
\end{proof}

The following two invariants of $F(\un{G})$ were defined in \cite[Def.1]{Bit2}:

\begin{definition} \label{i}
Let $R$ be a finite \'etale extension of $\CS$.
We define:
\begin{align*}
i(F(\un{G})) := \left \{ \begin{array}{l l}
{_m}\Br(R)                                           &  F(\un{G}) = \text{Res}_{R/\CS}(\un{\mu}_m) \\
\ker({_m}\Br(R) \xrightarrow{N^{(2)}} {_m}\Br(\CS))  &  F(\un{G}) = \text{Res}^{(1)}_{R/\CS}(\un{\mu}_m)
\end{array}\right.
\end{align*}
where for a group $*$, ${_m}*$ stands for its $m$-torsion~part,
and $N^{(2)}$ is induced by the norm map $N_{R/\CS}$. \\
For $F(\un{G}) = \prod_{k=1}^r F(\un{G})_k$ where each $F(\un{G})_k$ is one of the above,
$i(F(\un{G})):=\prod_{k=1}^r i(F(\un{G})_k)$. \\
We also define for such $R$:
\begin{align*}
j(F(\un{G})) := \left \{ \begin{array}{l l}
\Pic(R)/m                                                          &  F(\un{G}) = \text{Res}_{R/\CS}(\un{\mu}_m) \\
\ker ( \Pic(R)/m \xrightarrow{N^{(1)}/m} \Pic(\CS)/m )  &  F(\un{G}) = \text{Res}^{(1)}_{R/\CS}(\un{\mu}_m)
\end{array}\right.
\end{align*}
where $N^{(1)}$ is induced by $N_{R/\CS}$, and again $j(\prod_{k=1}^r F(\un{G})_k) := \prod_{k=1}^r j(F(\un{G})_k)$.
\end{definition}

\begin{definition} \label{admissible}
We call $F(\un{G})$ \emph{admissible}
if it is a finite direct product of factors of the form: 
\begin{itemize}
\item[(1)] $\text{Res}_{R/\CS}(\un{\mu}_m)$,
\item[(2)] $\text{Res}^{(1)}_{R/\CS}(\un{\mu}_m), [R:\CS]$ is prime to $m$,
\end{itemize}
where $R$ is any finite \'etale extension of $\CS$.
\end{definition}

\begin{lem} \label{decomposition of H2 F admissible}
If $F(\un{G})$ is admissible then there exists a short exact sequence of abelian groups:
\begin{equation} \label{exact sequence}
1 \to j(F(\un{G})) \to H^2_\et(\CS,F(\un{G})) \xrightarrow{\ov{i}_*} i(F(\un{G})) \to 1.
\end{equation}
This sequence splits thus reads:
$H^2_\et(\CS,F(\un{G})) \cong j(F(\un{G})) \times i(F(\un{G}))$.
\end{lem}

\begin{proof}
This sequence was shown in \cite[Cor.~2.9]{Bit2} for the case $S$ is nonempty.
The proof based on applying \'etale cohomology to the related Kummer exact sequence is similar for $S=\emptyset$.
The splitting when $F(\un{G})$ is quasi-split was proved in \cite[Thm.~1.1]{GG}.
When $F(\un{G}) = \mathrm{Res}^{(1)}_{R/\CS}(\un{\mu}_m)$, $[R:\CS]$ prime to $m$, consider the exact diagram
obtained by \'etale cohomology applied to the Kummer exact sequences related to $\un{\mu}_m$ over $\CS$ and~$R$:
\begin{equation} \label{N^2 diagram}
\xymatrix{
1 \ar[r] & \Pic(R)/m  \ar[r] \ar[d]^{N^{(1)}/m} & H^2_\et(R,\un{\mu}_m) \ar[r]^{i_*} \ar[d]^{N^{(2)}} & \Br(R)[m]   \ar[r] \ar[d]^{N^{(2)}[m]} & 1 \\
1 \ar[r] & \Pic(\CS)/m \ar[r]       & H^2_\et(\CS,\un{\mu}_m) \ar[r]                      & \Br(\CS)[m] \ar[r]        & 1.
}
\end{equation}
The splitting of the two rows then implies the one in the assertion.
\end{proof}

As a result we have two bijections as pointed-sets: the first is $\text{gen}(\un{G}) \cong i(F(\un{G}))$;
the affine case shown in \cite[Cor.~3.2]{Bit2} holds as aforementioned for $S=\emptyset$ as well,
in which case $\Br(C)$ is trivial (\cite[Theorem~4.5.1.(v)]{CTS}) thus $\un{G}$ admits a single genus.
The second bijection is $\ClS(\un{G})\cong i(F(\un{G}))$ unless $\un{G}$ is anisotropic at $S$,
for which it does not have to be injective \cite[Prop.~4.1]{Bit2};
hence when $S$ is empty this bijection is guaranteed.
Combining Lemma \ref{H1G iso to H^2F} with Lemma \ref{decomposition of H2 F admissible} these form (unless $\un{G}$ is anisotropic at $S$)
an isomorphism of finite abelian groups:
\begin{equation}
H^1_\et(\CS,\un{G}) \cong j(F(\un{G})) \times i(F(\un{G})).
\end{equation}

\mk

\section{Twisted-forms} \label{twisted forms}
Before continuing with the classification of $\un{G}$-forms,
we would like to recall the following general construction due to Giraud and prove one related Lemma.
Let $R$ be a unital commutative ring.
A central exact sequence of \'etale $R$-group schemes:
\begin{equation} \label{short exact}
1 \to A \xrightarrow{i} B \xrightarrow{\pi} C \to 1
\end{equation}
induces by \'etale cohomology a long exact sequence of pointed-sets (\cite[III, Lemma~3.3.1]{Gir}):
\begin{equation} \label{LES}
1 \to A(R) \to B(R) \to C(R) \to H^1_\et(R,A) \xrightarrow{i_*} H^1_\et(R,B) \to H^1_\et(R,C)
\end{equation}
in which $C(R)$ acts ''diagonally`` on the elements of $H^1_\et(R,A)$ in the following way:
For $c \in C(R)$, a preimage $X$ of $c$ under $B \to C$ is a $A$-bitorsor, i.e., $X=bA=Ab$ for some $b \in B(R')$,
where $R'$ is a finite \'etale extension of $R$ (\cite[III, 3.3.3.2]{Gir}).
Then given $[P] \in H^1_\et(R,A)$:
\begin{equation} \label{wedge}
c*P = P \overset{A}{\wedge} X = (P \times X)/(pa,a^{-1}x). 
\end{equation}
The exactness of \eqref{LES} implies that $B(R) \xrightarrow{\pi} C(R)$ is surjective if and only if $\ker(i_*)=1$.
This holds true starting with any twisted form ${^P}B$ of $B$, $[P] \in H^1_\et(R,A)$.

\begin{lem} \label{TFAE1}
The following are equivalent:
\begin{itemize}
\item[(1)] the push-forward map $H^1_\et(R,A) \xrightarrow{i_*} H^1_\et(R,B)$ is injective,
\item[(2)] the quotient map ${^P}B(R) \xrightarrow{\pi} C(R)$ is surjective for any $[P] \in H^1_\et(R,A)$,
\item[(3)] the $C(R)$-action on $H^1_\et(R,A)$ is trivial.
\end{itemize}
\end{lem}

\begin{proof}
Consider the exact and commutative diagram (cf.~\cite[III, Lemma~3.3.4]{Gir})
$$\xymatrix{
B(R) \ar[r]^{\pi} & C(R) \ar[r] & H^1_\et(R,A)    \ar[d]_{\cong}^{\theta_P} \ar[r]^{i_*} & H^1_\et(R,B) \ar[d]^{r}_{\cong} \\
{^P} B(R) \ar[r]^{\pi} & C(R) \ar[r] & H^1_\et(R,{^P} A) \ar[r]^{i_*'}                          & H^1_\et(R,{^P} B),
}$$
where the map $i_*'$ is obtained by applying \'etale cohomology to
the sequence~\eqref{short exact}
while replacing $B$ by the twisted group scheme ${^P} B$,
and $\theta_P$ is the induced twisting bijection. \\
$(1) \Leftrightarrow (2)$: The map $i_*$ is injective if and only if $\ker(i_*')$ is trivial for any $A$-torsor $P$.
By exactness of the rows, this is condition~$(2)$.  \\
$(1) \Leftrightarrow (3)$:
By~\cite[Prop.~III.3.3.3(iv)]{Gir}, $i_*$ induces an injection of $H^1_\et(R,A) / C(R)$ into $H^1_\et(R,B)$.
Thus $i_*:H^1_\et(R,A) \to H^1_\et(R,B)$ is injective if and only if $C(R)$ acts on $H^1_\et(R,A)$ trivially.
\end{proof}

Following B. Conrad in \cite{Con2}, we denote the group of outer automorphisms of $\un{G}$ by $\Theta$.

\begin{prop} \label{Theta is Out} (\cite[Prop.~1.5.1]{Con2}).
Assume $\Phi$ spans $X_\BQ$ and that $(X_\BQ, \Phi)$ is reduced.
The inclusion $\Theta \subseteq \textbf{Aut}(\text{Dyn}(\un{G}))$
is an equality, if the root datum is adjoint or simply-connected,
or if $(X_\BQ, \Phi)$ is irreducible and $(\Z \Phi^\vee)^*/\Z \Phi$ is cyclic.
\end{prop}

\begin{remark} \label{non-cyclicity}
The only case of irreducible $\Phi$ in which the non-cyclicity in Proposition \ref{Theta is Out} occurs,
is of type $\text{D}_{2n} (n \geq 2)$, in which $(\BZ \Phi^\vee)^* / \BZ \Phi \cong (\Z/2)^2$ (cf. \cite[Example~1.5.2]{Con2}).
\end{remark}

\begin{remark} \label{trivial action}
Since $\un{G}$ is reductive, $\textbf{Aut}(\un{G})$ is representable as an $\CS$-group
and admits the short exact sequence of smooth $\CS$-groups (see \cite[XXIV, 3.10]{SGA3},\cite[\S 3]{Con3}):
\begin{equation}  \label{Aut G}
1 \to \un{G}^\ad \to \textbf{Aut}(\un{G}) \to \Theta \to 1.
\end{equation}
Applying \'etale cohomology we get the exact sequence of pointed-sets:
\begin{equation} \label{exact sequence of pointed sets}
\textbf{Aut}(\un{G})(\CS) \to \Theta(\CS) \to H^1_\et(\CS,\un{G}^\ad) \xrightarrow{i_*} H^1_\et(\CS,\textbf{Aut}(\un{G})) \to H^1_\et(\CS,\Theta)
\end{equation}
in which by Lemma \ref{TFAE1} the $\Theta(\CS)$-action is trivial on $H^1_\et(\CS,\un{G}^\ad)$
if and only if $i_*$ is injective, being equivalent to the surjectivity of
$$ ({^P}\textbf{Aut}(\un{G}))(\CS) = \textbf{Aut}({^P}\un{G})(\CS) \to \Theta(\CS)  $$
for all $[P] \in H^1_\et(\CS,\Theta)$ (this action is trivial inside each genus).
\end{remark}

It is a classical fact that $H^1_\et(\CS,\Aut(\un{G}))$ is in bijection with twisted forms of $\un{G}$
up to isomorphism (for a general statement of this correspondence, see \cite[\S 2.2.4]{CF}).
Therefore this pointed-set shall be denoted from now and on by $\TwG$.
This bijection is done by associating any twisted form $\un{H}$ of $\un{G}$
with the $\textbf{Aut}(\un{G})$-torsor $\textbf{Iso}(\un{G},\un{H})$.
If $\un{H}$ is an inner-form of $\un{G}$,
then $[H]$ belongs to $\text{Im}(i_*)$ in \eqref{exact sequence of pointed sets}.

Sequence \eqref{Aut G} splits, provided that $\un{G}$ is quasi-split (as in Section \ref{torsors classification}). 
Recall that $\un{G}$ admits an inner form $\un{G}_0$ which is quasi-split.
Then $\textbf{Aut}(\un{G}_0) \cong \un{G}_0^\ad \rtimes \Theta$
(the outer automorphisms group of the two groups are canonically isomorphic).
This implies by \cite[Lemma~2.6.3]{Gil} the decomposition
\begin{equation} \label{Gille0}
\textbf{Twist}(\un{G}_0) = H^1_\et(\CS,\textbf{Aut}(\un{G}_0)) =
\coprod_{[P] \in H^1_\et(\CS,\Theta)} H^1_\et(\CS,{^P}(\un{G}_0^\ad))/ \Theta(\CS)
\end{equation}
where the quotients are taken modulo the action \eqref{wedge} of $\Theta(\CS)$ on the ${^P}(\un{G}^\ad)$-torsors.
But $\textbf{Twist}(\un{G}_0) = \TwG$ and as $\un{G}_0$ is inner: $H^1_\et(\CS,{^P}(\un{G}_0^\ad)) = H^1_\et(\CS,{^P}(\un{G}^\ad))$,
hence \eqref{Gille0} can be rewritten as:
\begin{equation} \label{Gille}
\TwG = \coprod_{[P] \in H^1_\et(\CS,\Theta)} H^1_\et(\CS,{^P}(\un{G}^\ad))/ \Theta(\CS).
\end{equation}
The pointed-set $H^1_\et(\CS,\Theta)$ 
classifies \'etale extensions of $\CS$ whose automorphism group embeds into $\Theta$.
As all $H^1_\et(\CS,{^P}(\un{G}^\ad))$ are finite, $\TwG$ is finite.
Together with Lemma \ref{H1G iso to H^2F} we get:

\begin{prop} \label{TwG as a disjoint union}
If $\un{G}$ is not of type $\text{A}$ then:
$$ \TwG \cong \coprod_{[P] \in H^1_\et(\CS,\Theta)} H^2_\et(\CS,F({^P}(\un{G}^\ad)))/\Theta(\CS), $$
the $\Theta(\CS)$-action on each component is carried by Lemma \ref{H1G iso to H^2F}
from the one on $H^1_\et(\CS,{^P}(\un{G}^\ad))$,~cf.~\eqref{wedge}.
\end{prop}

\begin{cor}
When $S=\emptyset$, i.e., over $C$, any outer form of $\un{G}$ has a unique genus on which $\Theta(\CS)$ acts trivially,
hence one has (including for type $\text{A}$):
$$ \TwG \cong \coprod_{[P] \in H^1_\et(\CS,\Theta)} j(F({^P}(\un{G}^\ad))). $$
\end{cor}

Since $C$ is smooth, $\CS$ is a Dedekind ring and any finite \'etale covering of it is
the normalization of $\CS$ (or of $C$ when $S=\emptyset$) in some finite separable extension of $K$, which is
unramified outside $S$.
So we may look on the fundamental groups over the according extension of fields;
The following is the list of all types of absolutely almost-simple $K$-groups (e.g., \cite[p.333]{PR}):
\mk
{
\begin{center} {\small
 \begin{tabular}{|c | c | c | }
 \hline
 Type of $G$             & $F(G^\ad)$              & $\textbf{Aut}(\text{Dyn}(\un{G}))$ 
\\ \hline  \hline
  ${^1}\text{A}_{n-1>0}$        & $\mu_n$                 & $\un{\Z/2}$   \\ \hline
	${^2}\text{A}_{n-1>0}$        & $R_{L/K}^{(1)}(\mu_n)$  & $\un{\Z/2}$   \\ \hline   	
  $\text{B}_n,\text{C}_n,\text{E}_7$          & $\mu_2$                 & $0$   \\ \hline
	${^1}\text{D}_n$   & \begin{tabular}[x]{@{}c@{}} $\mu_4, \ n=2k+1$ \\ $\mu_2 \times \mu_2, \ n=2k$ \end{tabular} & $\un{\Z/2}$ \\ \hline
	${^2}\text{D}_n$   & \begin{tabular}[x]{@{}c@{}} $R_{L/K}^{(1)}(\mu_4), \ n=2k+1$ \\ $R_{L/K}(\mu_2), \ n=2k$ \end{tabular} & $\un{\Z/2}$ \\ \hline
  ${^{3,6}}\text{D}_4$        & $R_{L/K}^{(1)}(\mu_2)$   & $\un{S_3}$   \\ \hline
	${^1}\text{E}_6$            & $\mu_3$                  & $\un{\Z/2}$  \\ \hline
	${^2}\text{E}_6$            & $R_{L/K}^{(1)}(\mu_3)$   & $\un{\Z/2}$  \\ \hline
  $\text{E}_8,\text{F}_4,\text{G}_2$        & $1$                      & $0$     \\ \hline
\end{tabular}
\label{table}
}\end{center}}

\bk

\section{Split fundamental group} \label{Split fundamental group}
In the following we show Proposition \ref{TwG as a disjoint union}.
We start with the simple case in which $\Theta=0$:

\begin{cor} \label{Theta is trivial}
If $\un{G}$ is of the type $\text{B}_{n>1},\text{C}_{n>1},\text{E}_7,\text{E}_8,\text{F}_4,\text{G}_2$,
for which $F(\un{G}^\ad) \cong \un{\mu}_m$ there exists an isomorphism of finite abelian groups
$\TwG \cong \Pic(\CS)/m \times {_m}\Br(\CS)$.
\end{cor}

\begin{proof}
As $\un{G}$ is not of type $\text{A}$ this derives from Proposition \ref{TwG as a disjoint union},
together with the fact that $\Theta(\CS)=0$ whence there is a single component $H^1_\et(\CS,\un{G}^\ad)$ on which the action of $\Theta(\CS)$ is trivial,
and the description of the isomorphic group $H^2_\et(\CS,F(\un{G}^\ad))$ is as in the split case in Lemma \ref{decomposition of H2 F admissible}.
\end{proof}

\begin{example} \label{SOq}
Given a regular quadratic $\CS$-form $Q$ of rank $2n+1$,
its \emph{special orthogonal group} $\un{G}=\un{\mathrm{SO}}_Q$
is smooth and connected of type $\text{B}_n$ (\cite[Thm.~1.7]{Con1}).
Since $F(\un{G}) = \un{\mu}_2$ we assume $\text{char}(K)$ is odd.
According to Corollary \ref{Theta is trivial} we then get
$$ \TwG \cong \Pic(\CS)/2 \times {_2}\Br(\CS). $$
In case $|S|=1$ and $Q$ is split by an hyperbolic plane,
an algorithm producing explicitly the inner forms of $Q$ is provided in \cite[Algorithm1]{Bit1}.
\end{example}

\mk

\section{Quasi-split fundamental group} \label{Quasi-split fundamental group}
Unless $\un{G}$ is of absolute type $\text{D}_4$,
$\Theta$ is either trivial or equals $\{\text{id},\tau:A \mapsto (A^{-1})^t) \}$. 
In the latter case, $\tau$ acts on the $\un{G}^\ad$-torsors via $X=\un{G}^\ad b$,
where $b$ is an outer automorphism of $\un{G}$, defined over some finite \'etale extension of $\CS$ (see \eqref{wedge}).
In particular:
$$ \tau* \un{G}^\ad=(\un{G}^\ad \times X) /(ga,a^{-1}x), $$
which is the opposite group $(\un{G}^\ad)^\op$, as the action is via $a^{-1}x=x(a^t)$,
($a$ is viewed as an element of $\un{G}^\ad$, not as an inner automorphism).
Now if $\tau$ is defined over $\CS$, then $\textbf{Aut}(\un{G})(\CS) \to \Theta(\CS)$ is surjective
and $(\un{G}^\ad)^\op$ is $\CS$-isomorphic to $\un{G}^\ad$,
hence as $\tau$ is the only non-trivial element in $\Theta(\CS)$,
the map $H^1_\et(\CS,{^P}\Aut(\un{G})) \to H^1_\et(\CS,\Theta)$ is surjective for all $[P] \in H^1_\et(\CS,\un{\Theta})$.
This implies by Remark \ref{trivial action} that $\Theta(\CS)$ acts trivially on $H^1_\et(\CS,\un{G}^\ad)$.
Otherwise, $\un{G}^\ad$ and $(\un{G}^\ad)^\op$ represent two distinct classes in $H^1_\et(\CS,\un{G}^\ad)$, being identified by $\Theta(\CS)$.

\mk

For any extension $R$ of $\CS$ and $L$ of $K$, we denote $\un{G}_R := \un{G} \otimes_{\CS} R$ and $G_L := G \otimes_K L$, respectively.
Let $[A_{\un{G}}]$ be the Tits class of the universal covering $\un{G}^\syc$ of $\un{G}$ (see Definition \ref{Tits algebra}).
This class does not depend on the choice of the representation $\rho$ of $\un{G}^\syc$,
thus its notation is omitted.
Recall that when $F(\un{G})$ splits $w_{\un{G}^\ad}$ defined in \ref{w_G}
coincides with $\Lm_* \circ \dl_{\un{G}}$.
Similarly, when $F(\un{G}) = \text{Res}_{R/\CS}(\un{\mu}_m)$ (\emph{quasi-split}) where $R/\CS$ is finite \'etale,
$\Lm_* \circ \dl_{\un{G}_R}$ and $w_{\un{G}_R^\ad}$ defined over $R$, coincide.

\begin{prop} \label{2 torsion in Br}
Suppose $\Theta \cong \un{\Z/2}$ and that $F(\un{G}^\ad) = \text{Res}_{R/\CS}(\un{\mu}_m)$,
$R$ is finite \'etale over $\CS$. 
Then TFAE:
\begin{itemize}
\item[(1)] $\un{G}_R$ admits an outer automorphism,
\item[(2)] $[A_{\un{G}_R}]$ is $2$-torsion in ${_m}\Br(R)$,
\item[(3)] $\Theta(R)$ acts trivially on $H^1_\et(R,\un{G}^\ad_R)$.
\end{itemize}
If, furthermore, $\un{G}$ is not of type $\text{A}$, or $S=\emptyset$,
then these facts are also equivalent to: 
\begin{itemize}
\item[(4)] $\un{G}$ admits an outer automorphism,
\item[(5)] $[A_{\un{G}}]$ is $2$-torsion in ${_m}\Br(R)$,
\item[(6)] $\Theta(\CS)$ acts trivially on $H^1_\et(\CS,\un{G}^\ad)$.
\end{itemize}
\end{prop}

\begin{proof}
By Lemma \ref{Tits class} the map $\Lm_* \circ \dl_{\un{G}^\ad _R} : H^1_\et(R,\un{G}^\ad _R) \to \Br(R)$ maps
$[\un{H}^\ad]$ to $[A_{\un{H}}] - [A_{\un{G} _R}]$ where $[A_{\un{H}}]$ is the Tits class of $\un{H}^\syc$ for a $\un{G}^\ad _R$-torsor $\un{H}^\ad$.
Consider this combined with the long exact sequence obtained by applying \'etale cohomology to the sequence \eqref{Aut G} tensored with $R$:
\begin{equation} \label{inner via outer}
\xymatrix{
                         &                  &  \text{Cl}_R(\un{G}^\ad_R) \ar@{^{(}->}[d]   \\
\Aut(\un{G}_R)(R) \ar[r] & \Theta(R) \ar[r] &  H^1_\et(R,\un{G}^\ad _R)  \ar[r]^-{i_*} \ar@{->>}[d]^{w_{\un{G}^\ad_R} = \Lm_* \circ \dl_{\un{G}^\ad_R}} & \textbf{Twist}(\un{G}_R)  \\
                         &                  &  {_m}\Br(R)
}
\end{equation}
where $\text{Cl}_R(\un{G}^\ad_R)$ is the principal genus of $\un{G}^\ad_R$
(see \cite[Prop.~3.1]{Bit2}) noting that $F(\un{G}_R^\ad) = \un{\mu}_m$).
Being an inner form of $\un{G}^\ad_R$,
$(\un{G}^\ad_R)^\op$ is obtained by
a representative in $H^1_\et(R,\un{G}^\ad_R)$.
Its $w_{\un{G}^\ad_R}$-image: $[A_{\un{G}_R}^\op]-[A_{\un{G}_R}]$ is trivial
if and only if $A_{\un{G}_R}$ is of order $\leq 2$ in ${_m}\Br(R)$,
which is equivalent to $\Aut(\un{G}_R)(R)$ surjecting on $\Theta(R)$,
and $\Theta(R)$ acting trivially on $H^1_\et(R,\un{G}_R^\ad)$
(see at the beginning of Section \ref{Quasi-split fundamental group}).

\mk

If, furthermore, $\un{G}$ is not of type $\text{A}$ or $S=\emptyset$, then by Lemma \ref{H1G iso to H^2F},
together with the Shapiro Lemma we get the isomorphisms of abelian groups:
\begin{equation} \label{H1 id}
H^1_\et(\CS,\un{G}^\ad) \cong H^2_\et(\CS,F(\un{G}^\ad)) \cong H^2_\et(R,\un{\mu}_m) \cong H^1_\et(R,\un{G}^\ad_R).
\end{equation}
So if $\Theta(R)$ acts trivially on $H^1_\et(R,\un{G}_R^\ad)$, then so does $\Theta(\CS)$ on $H^1_\et(\CS,\un{G}^\ad)$.
On the other hand if it does not, this implies that $\Aut(\un{G}_R)(R) \to \Theta(R) \cong \un{\Z/2}$ is not surjective,
thus neither is $\Aut(\un{G})(\CS) \to \Theta(\CS)$,
which is equivalent to $\Theta(\CS)$ acting non-trivially on $H^1_\et(\CS,\un{G}^\ad)$ by Remark \ref{trivial action}.
Moreover, since $i(F(\un{G}^\ad_R)) = i(F(\un{G}^\ad)) = {_m}\Br(R)$ (Def. \eqref{i}), the identification \eqref{H1 id}
shows that $\text{Cl}_R(\un{G}^\ad_R)$ bijects to $\ClS(\un{G}^\ad)$,
whence $[A_{\un{G}_R}]$ is $2$-torsion in ${_m}\Br(R)$ if and only if $[A_{\un{G}}]$ is.
\end{proof}

If we wish to interpret a $\un{G}$-torsor as a twisted form of some basic form,
we shall need to describe $\un{G}$ first as the automorphism group of such an $\CS$-form.

\begin{example} \label{PGLn}
Let $A$ be a division $\CS$-algebra of degree $n > 2$.
Then $\un{G} = \un{\textbf{SL}}(A)$ of type $\text{A}_{n-1>1}$ is smooth and connected (\cite[Lemma~3.3.1]{Con2}).
It admits a non-trivial outer automorphism~$\tau$.
If the transpose anti-automorphism $A \cong A^\op$ is defined over $\CS$ (extending $\tau$ by inverting again),
then $\tau~\in~\textbf{Aut}(\un{G})(\CS)$.
Otherwise, as $(\un{G}^\ad)^\op$ is not $\CS$-isomorphic by some conjugation to $\un{G}^\ad = \un{\textbf{PGL}}(A)$,
it represents a non-trivial class in $H^1_\et(\CS,\un{G}^\ad = \textbf{Inn}(\un{G}))$,
whilst its image in $\TwG$ is trivial by the inverse isomorphism $x \mapsto x^{-1}$
defined over $\CS$ (say, by the Cramer rule).
So finally $\Theta(\CS)$ acts trivially on $H^1_\et(\CS,\un{\textbf{PGL}}(A))$
if and only if $\text{ord}(A) \leq 2$ in $\Br(\CS)$, as Prop. \ref{2 torsion in Br} predicts.
\end{example}

\subsection{Type $\text{D}_{2k}$} \label{D2k}
Let $A$ be an Azumaya $\CS$-algebra ($\text{char}(K) \neq 2$) of degree $2n$
and let $(f,\s)$ be a \emph{quadratic pair} on $A$, namely, $\s$ is an 
involution on $A$ 
and $f:\text{Sym}(A,\s) = \{ x \in A:\s(x)=x \} \to \CS$ is a linear map.
The scalar $\mu(a) := \s(a) \cdot a$ is called the \emph{multiplier} of $a$.
For $a \in A^\times$ we denote by $\textbf{Int}(a)$ the induced inner automorphism.
If $\s$ is orthogonal, the associated \emph{similitude group} is: 
$$ \un{\textbf{GO}}(A,f,\s) := \{ a \in A^\times: \mu(a) \in \CS^\times, \ f \circ \textbf{Int}(a) = f \},  $$
and the map $a \mapsto \textbf{Int}(a)$ is an isomorphism of the \emph{projective similitude group} $\un{\textbf{PGO}}(A,f,\s) := \textbf{GO}(A,f,\s)/\CS^\times$
with the group of rational points $\textbf{Aut}(A,f,\s)$.
Such a similitude is said to be \emph{proper} if the induced automorphism
of the Clifford algebra $C(A,f,\s)$ is the identity on the center; otherwise it is said to be \emph{improper}.
The subgroup $\un{G}=\textbf{PGO}^+(A,f,\s)$ of these proper similitudes is connected and adjoint,  
called the \emph{projective special similitude group}.
If the discriminant of $\s$ is a square in $\CS^\times$, 
then $\un{G}$ is of type ${^1}\text{D}_n$.
Otherwise of type ${^2}\text{D}_n$.

\mk

When $n=2k$, in order that $\Theta$ captures the full structure of $\textbf{Aut}(\text{Dyn}(\un{G}))$,
we would have to restrict ourselves to the two edges of simply-connected
and adjoint groups (see Remark \ref{non-cyclicity}).

\begin{cor} \label{2D2k}
Let $\un{G}$ be of type ${^2}\text{D}_{2k},k\neq 2$, simply-connected or adjoint. \\
For any $[P] \in H^1_\et(\CS,\un{\Z/2})$ let $R_P$ be the corresponding quadratic \'etale extension of $\CS$.
Then:
$$ \TwG \cong \coprod_{[P] \in H^1_\et(\CS,\un{\Z/2})} \Pic(R_P)/2 \times {_2}\Br(R_P). $$
\end{cor}

\begin{proof}
Any form ${^P}(\un{G}^\ad)$ has Tits class $[A_{{^P}\un{G}}]$ of order $\leq 2$ in ${_2}\Br(R_P)$. 
Hence as $\Theta \cong \un{\Z/2}$ and $\un{G}$ is not of type $\text{A}$,
by Proposition \ref{2 torsion in Br} $\Theta(\CS)$ acts trivially on $H^1_\et(\CS,{^P}(\un{G}^\ad))$
for all $P$ in $\Theta(\CS)$.
All fundamental groups are admissible, so the Corollary statement is Proposition \ref{TwG as a disjoint union}
together with the description of each $H^2_\et(\CS,F({^P}(\un{G}^\ad)))$ as in Lemma~\ref{decomposition of H2 F admissible}.
\end{proof}

\mk

\section{Non quasi-split fundamental group} \label{Non quasi-split fundamental group}
When $F(\un{G}^\ad)$ is not quasi-split, we cannot apply the Shapiro Lemma as in \eqref{H1 id}
to gain control on the action of $\Theta(\CS)$ on $H^1_\et(\CS,F(\un{G}^\ad))$.
Still under some conditions this action is provided to be trivial.

\begin{remark} \label{trivial Br}
As opposed to ${_m}\Br(K)$ which is infinite for any integer $m > 1$, ${_m}\Br(\CS)$ is finite.
To be more precise, if $S \neq \emptyset$, $\Sp(\CS)$ is obtained by removing $|S|$ points from the projective curve $C$,
hence $|{_m}\Br(\CS)| = m^{|S|-1}$ (see the proof of \cite[Cor.~3.2]{Bit2}).
When $S=\emptyset$ we have $\Br(C)=1$. 
In particular, if $\un{G}$ is not of absolute type $\text{A}$ and $F(\un{G}^\ad)$ splits over an extension $R$
such that the number of places in $\text{Frac}(R)$ which lie above places in $S$ is $1$, or when $S=\emptyset$,
then $\un{G}^\ad$ can posses only one genus and consequently the $\Theta(\CS)$-action on $H^1_\et(\CS,\un{G}^\ad)$ is trivial.  
\end{remark}

E. Artin in \cite{Art} calls a Galois extension $L$ of $K$ \emph{imaginary}
if no prime of $K$ is decomposed into distinct primes in $L$.
We shall similarly call a finite \'etale extension of $\CS$ \emph{imaginary} if no prime of $\CS$ is decomposed into distinct primes in it.

\begin{lem} \label{mBr equal R imaginary}
If $R$ is imaginary over $\CS$ and $m$ is prime to $[R:\CS]$, then ${_m}\Br(R) = {_m}\Br(\CS)$.
\end{lem}

\begin{proof}
If $S=\emptyset$ and $R/C$ is imaginary then $\Br(R)=\Br(\CS)=1$.
Otherwise, the composition of the induced norm $N_{R/\CS}$ with the diagonal morphism coming from the Weil restriction
\begin{equation*} 
\un{\BG}_{m,\CS} \to \text{Res}_{R/\CS}(\un{\BG}_{m,R}) \xrightarrow{N_{R/\CS}} \un{\BG}_{m,\CS}
\end{equation*}
is the multiplication by $n := [R:\CS]$.
It induces together with the Shapiro Lemma the maps:
$$ H^2_\et(\CS,\un{\BG}_{m,\CS}) \to H^2_\et(R,\un{\BG}_{m,R}) \xrightarrow{N^{(2)}} H^2_\et(\CS,\un{\BG}_{m,\CS}) $$
whose composition is the multiplication by $n$ on $H^2_\et(\CS,\un{\BG}_{m,\CS})$.
Identifying $H^2_\et(*,\un{\BG}_m)$ with $\Br(*)$ and restricting to the $m$-torsion subgroups gives the composition
$$ {_m}\Br(\CS) \to {_m}\Br(R) \xrightarrow{N^{(2)}}{_m}\Br(\CS) $$
being still multiplication by $n$, 
thus an automorphism when $n$ is prime to $m$. 
This means that ${_m}\Br(\CS)$ is a subgroup of ${_m}\Br(R)$.
As $R$ is imaginary over $\CS$, it is obtained by removing $|S|$ points from the projective curve defining its fraction field,
so $|{_m}\Br(R)| = |{_m}\Br(\CS)|=m^{|S|-1}$ by Remark \ref{trivial Br}, and the assertion follows.
\end{proof}

\begin{cor} \label{imaginary same torsion}
If $F(\un{G}) = \text{Res}^{(1)}_{R/\CS}(\un{\mu}_m)$ is admissible and $R/\CS$ is imaginary,
then $i(F(\un{G}))=\ker({_m}\Br(R) \to {_m}\Br(\CS))$ (see Def. \ref{i}) is trivial,
hence $\un{G}$ admits a single genus (cf. \cite[Cor.~3.2]{Bit2}).
\end{cor}

\subsection{Type $\text{E}_6$}
A \emph{hermitian} Jordan triple over $\CS$
is a triple $(A,\fX,U)$ consisting of a quadratic \'etale $\CS$-algebra $A$ with conjugation $\s$,
a free of finite rank $\CS$-module $\fX$,
and a quadratic map $U:\fX \to \Hom_A(\fX^\s,\fX): x \mapsto U_x$,
where $\fX^\s$ is $\fX$ with scalar multiplication twisted by $\s$,
such that $(\fX,U)$ is an (ordinary) Jordan triple as in \cite{McC}.
In particular if $\fX$ is an \emph{Albert} $\CS$-algebra, then it is called an \emph{hermitian Albert triple}.
In that case the associated trace form $T:A \times A \to \CS$ is symmetric non-degenerate and it follows 
that the structure group of $\fX$ agrees with its group of norm similarities.
Viewed as an $\CS$-group, it is reductive with center of rank $1$
and its semisimple part, which we shortly denote $G(A,\fX)$, is simply connected of type $\text{E}_6$.  
It is of relative type ${^1}\text{E}_6$ if $A \cong \CS \times \CS$ and of type ${^2}\text{E}_6$ otherwise. 

\mk

Groups of type ${^1}\text{E}_6$ are classified by four relative types,
among them only ${^1}\text{E}_{6,2}^{16}$ has a non-commutative Tits algebra,
thus being the only type in which $\Theta(\CS) \cong \Z/2$ may act non-trivially on $H^1_\et(\CS,\un{G}^\ad)$.
More precisely, the Tits-algebra in that case is a division algebra $D$ of degree $3$ (cf. \cite[p.58]{Tits66})
and the $\Theta(\CS)$-action is trivial if and only if $\text{ord}([D]) \leq 2$ in $\Br(\CS)$.
But $\text{ord}([D])$ is odd, thus this action is trivial if and only if $D$ is a matrix $\CS$-algebra.

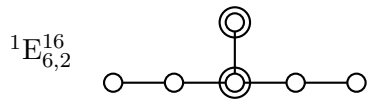
\mk

\begin{center}
  \begin{tikzpicture}[scale=.4]
    \draw (-1,1) node[anchor=east]  {${^1}\text{E}_{6,2}^{16}$};
    \foreach \x in {0,...,4}
    \draw[thick,xshift=\x cm] (\x cm,0) circle (3 mm);
    \foreach \y in {0,...,3}
    \draw[thick,xshift=\y cm] (\y cm,0) ++(.3 cm, 0) -- +(14 mm,0);
    \draw[thick] (4 cm,2 cm) circle (3 mm) circle(5 mm);
    \draw[thick] (4 cm, 3mm) -- +(0, 1.4 cm);
    \draw[thick] (4 cm, 0) circle(5 mm);
  \end{tikzpicture}
\end{center}

In the case of type ${^2}\text{E}_6$, one has six relative types (cf. \cite[p.59]{Tits66}),
among which only ${^2}E_{6,2}^{16''}$ has a non-commutative Tits algebra (cf. \cite[p.211]{Tits71}).
Its Tits algebra is a division algebra of degree $3$ over $R$, and its Brauer class has trivial corestriction in $\Br(\CS)$.
By Albert and Riehm, this is equivalent to $D$ possessing an $R/\CS$-involution.

\mk

 \begin{center}
  \begin{tikzpicture}[scale=.4]
     \draw (-1,1) node[anchor=east]  {${^2}\text{E}_{6,2}^{16''}$};
    \draw[thick] (0,0) circle (3 mm) circle (5 mm);
    \draw[thick] (2 cm ,0) circle (3 mm) circle(5 mm);
   \draw[thick] (4 cm ,1 cm) circle (3 mm);
    \draw[thick] (6 cm,1 cm) circle (3 mm);
     \draw[thick] (4 cm,-1 cm) circle (3 mm);
     \draw[thick] (6 cm,-1 cm) circle (3 mm);
    \draw[thick] (.3 cm,0) -- (1.7 cm, 0);
    \draw[thick] (2.3 cm,.1 cm) -- (3.7 cm, .9 cm);
    \draw[thick] (4.3 cm,1 cm) -- (5.7 cm, 1 cm);
    \draw[thick] (2.3 cm, -.1 cm) -- (3.7 cm, -.9 cm);
    \draw[thick] (4.3 cm, - 1 cm) -- (5.7 cm, - 1cm);
  \end{tikzpicture}
\end{center}

From now on $\sim$ denotes the equivalence relation on the Brauer group which identifies the class of an Azumaya
algebra with the class of its opposite.

\begin{cor} \label{E6}
Let $\un{G}$ be of (absolute) type $\text{E}_6$. 
For any $[P] \in H^1_\et(\CS,\un{\Z/2})$ let $R_P$ be the corresponding quadratic \'etale extension of $\CS$. 
Then
\begin{align*}
\TwG & \cong \Pic(\CS)/3 \times {_3}\Br(\CS)/\sim \\ \nonumber
     & \coprod_{1 \neq [P]} \ker(\Pic(R_P)/3 \to \Pic(\CS)/3) \times (\ker({_3}\Br(R_P) \to {_3}\Br(\CS)))/\sim,
\end{align*}
where $[P]$ runs over $H^1_\et(\CS,\un{\Z/2})$.
The relation $\sim$ is trivial in the first component unless $\un{G}^\ad$ is of type ${^1}\text{E}_{6,2}^{16}$
and is trivial in the other components unless ${^P}(\un{G}^\ad)$ is of type ${^2}E_{6,2}^{16''}$.
\end{cor}

\begin{proof}
The group $\Theta(\CS)$ acts trivially on members of the same genus,
so it is sufficient to check its action on the set of genera for each type.
Since $F({^P}(\un{G}^\ad))$ is admissible for any $[P] \in H^1_\et(\CS,\Theta)$,
by \cite[Cor.~3.2]{Bit2} the set of genera of each ${^P}(\un{G}^\ad)$
bijects as a pointed-set to $i(F({^P}(\un{G}^\ad)))$,
so the assertion is Proposition \ref{TwG as a disjoint union}
together with Lemma~\ref{decomposition of H2 F admissible}.
The last claims are retrieved from the above discussion on the trivial action of $\Theta(\CS)$
when ${^P}(\un{G}^\ad)$ is not of type ${^1}\text{E}_{6,2}^{16}$ or ${^2}E_{6,2}^{16''}$.
\end{proof}

\begin{example} \label{Example E6}
Let $C$ be the elliptic curve $Y^2Z = X^3 + XZ^2 + Z^3$ defined over $\BF_3$.
Then:
$$ C(\BF_3) = \{ (1:0:1), (0:1:2), (0:1:1), (0:1:0) \}. $$
Removing the $\BF_3$-point $\iy = (0:1:0)$ the obtained smooth affine curve $C^\af$ is $y^2 = x^3 + x + 1$.
Letting $\Oiy = \BF_3[C^\af]$ we have $\Pic(\Oiy) \cong C(\BF_3)$ (e.g., \cite[Example~4.8]{Bit1}).
Among the affine supports of points in $C(\BF_3) - \{ \iy \}$:
$$ \{ (1,0), (0,1/2) = (0,2), (0,1) \}, $$
only $(1,0)$ has a trivial $y$-coordinate thus being of order $2$ (according to the group law there),
to which corresponds the fractional ideal $P = ( x-1,y )$ of order $2$ in $\Pic(\Oiy)$. 
As $\Pic(\Oiy)/3=1$ and $\Br(\Oiy)=1$, 
a form of type ${^1}\text{E}_6$ has no non-isomorphic \emph{inner} form,
while a form of type ${^2}\text{E}_6$ may have more; 
for example $R=\Oiy \oplus P$ being geometric and \'etale cannot be imaginary over $\Oiy$,
which means it is obtained by removing two points from a projective curve, thus
$$ \ker({_3}\Br(R) \to {_3}\Br(\Oiy))/\sim = {_3}\Br(R)/\sim = \{[R],[A],[A^\text{op}] \}/\sim = \{[R],[A]\}, $$
hence an $\Oiy$-group of type $\mathrm{E}_6$ splitting over $R$ admits a non-isomorphic inner form. 
\end{example}

\subsection{Type $\text{D}_{2k+1}$}
Recall from Section \ref{D2k} that an adjoint $\CS$-group $\un{G}$ of absolute type $\text{D}_{n}$
can be realized as $\un{\textbf{PGO}}^+(A,\s)$ where $A$ is Azumaya of degree $2n$ and $\s$ is an orthogonal involution on $A$.
Suppose $n$ is odd. 
If $\un{G}$ is of relative type ${^1}\text{D}_n$ then $F(\un{G}) = \un{\mu}_4$ is admissible,
thus not being of absolute type $\text{A}$, $\ClS(\un{G})$ bijects to $j(\un{\mu}_4) = \Pic(\CS)/4$
and $\text{gen}(\un{G})$ bijects to $i(\un{\mu}_4) = {_4}\Br(\CS)$.
Otherwise, when $\un{G}$ is of type ${^2}\text{D}_n$,
then $F(\un{G}) = \text{Res}^{(1)}_{R/\CS}(\un{\mu}_4)$ where $R/\CS$ is quadratic.
Again not being of absolute type $\text{A}$, $\ClS(\un{G}) \cong j(F(\un{G})) = \ker(\Pic(R)/4 \to \Pic(\CS)/4)$,
but here, as $F(\un{G})$ is not admissible, by \cite[Cor.~3.2]{Bit2} $\text{gen}(\un{G})$
only injects in $i(F(\un{G}))=\ker({_4}\Br(R) \to {_4}\Br(\CS))$.
If $R/\CS$ is imaginary, then by Lemma \ref{mBr equal R imaginary} $i(F(\un{G}))=1$.
Altogether by Proposition \ref{TwG as a disjoint union} we get:

\begin{cor} \label{D 2k+1}
Let $\un{G}$ be of (absolute) type $\text{D}_{2k+1}$. 
For any $[P] \in H^1_\et(\CS,\un{\Z/2})$ let $R_P$ be the corresponding quadratic \'etale extension of $\CS$.
Then there exists an exact sequence of pointed-sets 
\begin{align*}
\TwG &\hookrightarrow \Pic(\CS)/4 \times \coprod_{1 \neq [P]} \ker(\Pic(R_P)/4 \to \Pic(\CS)/4)    \\ \nonumber
     &\times \left({_4}\Br(\CS)/\sim \times \coprod_{1 \neq [P]} (\ker({_4}\Br(R_P) \to {_4}\Br(\CS)))/\sim \right),
\end{align*} 
where $[P]$ runs over $H^1_\et(\CS,\un{\Z/2})$ and $[A] \sim [A^\op]$.
This map surjects onto the first component.
Whenever $R_P/\CS$ is imaginary $\ker({_4}\Br(R_P) \to {_4}\Br(\CS))=1$ and this map is a bijection. 
\end{cor}

\begin{example}
Let $\Oiy = \BF_q[x]$ ($q$ is odd) obtained by removing $\iy = (1/x)$ from the projective line over $\BF_q$. 
Suppose $q \in 4\N-1$ so $-1 \notin \BF_q^2$,
and let $\un{G} = \un{\textbf{SO}}_{10}$ be defined over $\Oiy$.
The discriminant of an orthogonal form $Q_B$ induced by an $n \times n$ matrix $B$ is $\text{disc}(Q_B) = (-1)^{\fc{n(n-1)}{2}}\det(B)$.
As $\text{disc}(Q_{1_{10}}) = -1$ is not a square in $\Oiy$, $\un{G}$ is considered of type ${^2}\text{D}_5$.
It admits a maximal torus $\un{T}$ containing five $2 \times 2$ rotations blocks
$\left( \begin{array}{cc}
     a & b   \\
    -b & a \\
\end{array}\right): a^2 +b^2 =1$
on the diagonal.
Over $R=\Oiy[i]$ such block is diagonalizable;
it becomes $\diag(t,t^{-1})$. 
The obtained diagonal torus $\un{T}_s' = P \un{T}_s P^{-1}$
where $\un{T}_s = \un{T} \otimes R$ and $P$ is some invertible $10 \times 10$ matrix over $R$,
is split and $5$-dimensional, so may be identified with the $5 \times 5$ diagonal torus,
whose positive roots are:
$$ \a_1 = \ve_1-\ve_2, \ \a_2 =\ve_2-\ve_3, \ \a_3 =\ve_3-\ve_4, \ \a_4 = \ve_4 - \ve_5, \ \a_5 = \ve_4 + \ve_5. $$ 
Let $g$ be the matrix differing from the $10 \times 10$ unit only at the last $2 \times 2$ block,
being
$\left( \begin{array}{cc}
     0 & 1   \\
     1 & 0 \\
\end{array}\right)$.
Then $\det(g)=-1$ thus $\text{disc}(Q_g)=1$ where $Q_g$ is the induced quadratic form.
This means that $\un{G}' = \un{\textbf{SO}}(Q_g)$ of type ${^1}\text{D}_5$ is the unique outer form of $\un{G}$
(up to $\CS$-isomorphism).
Then $\Theta = \textbf{Aut}(\text{Dyn}(G))$ acts on $\Lie(g \un{T}'_s g^{-1})$ by mapping the last block
$\left( \begin{array}{cc}
     0 & \ln(t)   \\
     -\ln(t) & 0 \\
\end{array}\right)$
to
$\left( \begin{array}{cc}
     0 & -\ln(t)   \\
     \ln(t) & 0 \\
\end{array}\right)$
and so swapping the above two roots $\a_4$ and $\a_5$. 
Since $\Oiy$ and $R$ are PIDs, their Picard groups are trivial. 
As only one point was removed in both domains also $\Br(\Oiy)=\Br(R)=1$.
We remain with only the two above forms, i.e., $\TwG = \{[\un{G}], [\un{G}']\}$.

\mk

The same holds for $\CS=\BF_q[x,x^{-1}]$:
again it is a UFD thus $\un{G} = \un{\textbf{SO}}_{10}$ defined over it still posseses only one non-isomorphic outer form.
As $\CS$ is obtained by removing two points from the projective $\BF_q$-line,
this time ${_4}\Br(\CS)$ is not trivial, but still equals ${_4}\Br(\CS)$, so:
$\ker({_4}\Br(R) \to {_4}\Br(\CS)) = 1$. 
\end{example}

\subsection{Type $\text{D}_4$}
This case deserves a special regard as $\Theta$ is the symmetric group $\un{S_3}$
when $\un{G}$ is adjoint or simply-connected (cf. Prop. \ref{Theta is Out}).
Suppose $C$ is an Octonion $\CS$-algebra with norm $N$.
For any similitude $t$ of $N$ (see Section \ref{D2k}) there exist similitudes $t_2$ and $t_3$ such that
$$ t_1(xy) = t_2(x) \cdot t_3(y) \ \ \forall x,y \in C. $$
Then the mappings:
\begin{align}
\a  : [t_1] \mapsto [t_2], \ \b  : [t_1] \mapsto [\hat{t}_3]
\end{align}
where $\hat{t}(x) := \mu(t)^{-1} \cdot t(x)$, satisfy $\a^2 = \b^3 = \text{id}$ and generate $\Theta = \textbf{Out}(\textbf{PGO}^+(N)) \cong \un{S_3}$.

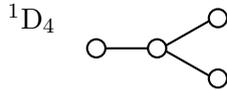
\begin{center}
  \begin{tikzpicture}[scale=.4]
     \draw (-1,1) node[anchor=east]  {${^1}\text{D}_{4}$};
   \draw[thick] (0,0) circle (3 mm);
    \draw[thick] (2 cm ,0) circle (3 mm);
   \draw[thick] (4 cm ,1 cm) circle (3 mm);
     \draw[thick] (4 cm,-1 cm) circle (3 mm);
  \draw[thick] (.3 cm,0) -- (1.7 cm, 0);
    \draw[thick] (2.3 cm,.1 cm) -- (3.7 cm, .9 cm);
    \draw[thick] (2.3 cm, -.1 cm) -- (3.7 cm, -.9 cm);
  \end{tikzpicture}
\end{center}
Having three conjugacy classes,
there are three classes of outer forms of $\un{G}$ (cf. \cite[p.253]{Con2}),
which we denote as usual by ${^1}\text{D}_4$, ${^2}\text{D}_4$ and ${^{3,6}}\text{D}_4$.
The groups in the following table are the generic fibers of these outer forms,
$L/K$ is the splitting extension of $F(G^\ad)$ (note that in the case ${^6}D_4$ $L/K$ is not Galois):

{
\begin{center} {\small
 \begin{tabular}{|c | c |c| }
 \hline
 Type of $G$           & $F(G^\ad)$              & \ \ $[L:K]$ \ \  \\ \hline  \hline
	${^1}D_4$            &  $\mu_2 \times \mu_2$   & 1  \\ \hline
	${^2}D_4$            & $R_{L/K}(\mu_2)$        & 2  \\ \hline
  ${^{3,6}}D_4$        & $R_{L/K}^{(1)}(\mu_2)$  & 3  \\ \hline
\end{tabular}
\label{table D4}
}\end{center}}

Starting with $\un{G}$ of type ${^1}\text{D}_4$,
one sees that $F({^P}(\un{G}^\ad))$ -- splitting over some corresponding extension $R / \CS$ --
is admissible for any $[P] \in H^1_\et(\CS,\Theta)$,
thus according to Lemma \ref{decomposition of H2 F admissible}
$$ \forall [P] \in H^1_\et(\CS,\Theta): \ \ H^2_\et(\CS,F({^P}(\un{G}^\ad))) \cong j(F({^P}(\un{G}^\ad))) \times i(F({^P}(\un{G}^\ad))).  $$
The action of $\Theta(\CS)$ is trivial on the first factor, classifying torsors of the same genus,
so we concentrate on its action on $i(F({^P}(\un{G}^\ad)))$.
Since $\Theta \not\cong \un{\Z/2}$ we cannot use Prop. \ref{2 torsion in Br},
but we may still imitate its arguments:

The group $\Theta(\CS)$ acts non-trivially on $H^1_\et(\CS,{^P}(\un{G}^\ad))$ for some $[P] \in H^1_\et(\CS,\Theta)$
if it identifies two non isomorphic torsors of ${^P}(\un{G}^\ad)$. 
The Tits algebras of their universal coverings lie in $({_2}\Br(\CS))^2$ if ${^P}(\un{G}^\ad)$ is of type ${^1}\text{D}_4$,
i.e., if $P$ belongs to the trivial class in $H^1_\et(\CS,\Theta)$,
in ${_2}\Br(R)$ for $R$ quadratic \'etale over $\CS$ if ${^P}(\un{G}^\ad)$ is of type ${^2}\text{D}_4$, i.e., if $[P] \in {_2}H^1_\et(\CS,\Theta)$,
and in $\ker({_2}\Br(R) \to {_2}\Br(\CS))$ for a cubic \'etale extension $R$ of $\CS$ if ${^P}(\un{G}^\ad)$ is one of the types ${^{3,6}}\text{D}_4$,
i.e., if $[P] \in {_3}H^1_\et(\CS,\Theta)$.
Therefore these Tits algebras must be $2$-torsion,
which means that the two torsors are $\CS$-isomorphic in the first case and $R$-isomorphic in the latter three.
If $F({^P}(\un{G}^\ad))$ is quasi-split this means (by the Shapiro Lemma)
that $\Theta(\CS)$ acts trivially on $H^1_\et(\CS,{^P}(\un{G}^\ad))$.
If $F({^P}(\un{G}^\ad))$ is not quasi-split,
according to Corollary \ref{imaginary same torsion} if $R$ is imaginary over $\CS$ then $i(F({^P}(\un{G}^\ad)))=1$.

\mk

If a quadratic form $Q$ has a trivial discriminant on a vector space $V$,
the Tits algebras of the group are $\End(V)$ and the two components $\C_+(Q), \C_-(Q)$ of the even Clifford algebra of $Q$,
and the triality automorphism cyclically permutes those three.
More generally, if the group is represented as $\textbf{PGO}^+(A,\sigma)$
for some orthogonal involution of trivial discriminant on a central simple algebra A of degree $8$,
triality automorphisms permute $A$ and the two components of the Clifford algebra $\C(A,\sigma)$;
Altogether we finally get:

\begin{cor}
Let $\un{G}$ be of (absolute) type $\text{D}_4$ being simply-connected or adjoint. \\
For any $[P] \in H^1_\et(\CS,\Theta)$ let $R_P$ be the corresponding \'etale extension of $\CS$.
Then:
\begin{align*}
\TwG &\cong   (\Pic(\CS)/2 \times {_2}\Br(\CS))^2                                          \\ \nonumber
     &\coprod_{1 \neq [P] \in {_2}H^1_\et(\CS,\Theta)} \Pic(R_P)/2 \times {_2}\Br(R_P)                                                  \\ \nonumber
		 &\coprod_{1 \neq [P] \in {_3}H^1_\et(\CS,\Theta)} \ker(\Pic(R_P)/2 \to \Pic(\CS)/2) \times (\ker({_2}\Br(R_P) \to {_2}\Br(\CS))) / \Theta(\CS).
\end{align*}
If $R_P$ is imaginary over $\CS$, then $\ker({_2}\Br(R_P) \to {_2}\Br(\CS))=1$.
\end{cor}

\mk

\section{The anisotropic case} \label{anisotropic case}
Now suppose that $\un{G}$ 
does admit a twisted form such that the generic fiber of its universal covering is anisotropic at $S$.
As previously mentioned, such group must be of absolute type $\text{A}$ and $S \neq \emptyset$.
Over a local field $k$, an outer form of a group of type ${^1}\text{A}$ 
which is anisotropic, must be the special unitary group arising by some hermitian form $h$
in $r$ variables over a quadratic extension of $k$ or over a quaternion $k$-algebra (\cite[\S 4.4]{Tits79}).

\mk

A \emph{unitary} $\CS$-group is $\un{\textbf{U}}(B,\s) := \textbf{Iso}(B,\s)$
where $B$ is a non-split quaternion Azumaya defined over an \'etale quadratic extension $R$ of $\CS$
and $\s$ is a unitary involution on $B$, i.e., whose restriction to the center $R$ is not the identity.
The \emph{special unitary group} is the kernel of the reduced norm:
$$ \un{\textbf{SU}}(B,\s) := \ker(\Nrd:\un{\textbf{U}}(B,\tau) \surj \un{\textbf{GL}}_1(R)). $$
These are of relative type ${^2}\text{C}_{2m}$ ($m \geq 2$) (\cite{Tits79}, loc. sit.)
and isomorphic over $R$ to type ${^1}\text{A}_{2m-1}$.

\mk

So in order to determine exactly when $H^1_\et(\CS,\un{G}^\syc)$ \emph{does not} vanish,
we may restrict ourselves to $\CS$-groups whose universal covering is either $\un{\textbf{SL}}_1(A)$
or $\un{\textbf{SU}}(B,\s)$.
In the first case, the reduced norm applied to the units of $A$ forms the short exact sequence of smooth $\CS$-groups:
\begin{equation} \label{SES SL(A)}
1 \to \un{\textbf{SL}}_1(A) \to \un{\textbf{GL}}_1(A) \xrightarrow{\Nrd}  \un{\BG}_m \to 1.
\end{equation}
Then \'etale cohomology gives rise to the long exact sequence:
\begin{equation} \label{LES SL(A)}
1 \to \CS^\times \big / \Nrd(A^\times) \to H^1_\et(\CS,\un{\textbf{SL}}_1(A))
                                          \xrightarrow{i_*} H^1_\et(\CS,\un{\textbf{GL}}_1(A)) \xrightarrowdbl{\Nrd_*} H^1_\et(\CS,\un{\BG}_m) \cong \Pic(\CS)
\end{equation}
in which $\Nrd_*$ is surjective since $\un{\textbf{SL}}_1(A)$
is simply-connected and $\CS$ is of Douai-type (see above).

\begin{definition}
We say that the \emph{local-global Hasse principle} holds for $\un{G}$ if $h_S(\un{G})=|\ClS(\un{G})|=1$.
\end{definition}

Thus the Hasse principle says that an $\CS$-group is $\CS$-isomorphic to $\un{G}$
if and only if it is $K$-isomorphic to it. 
This is automatic for simply-connected groups 
which are not of type $\text{A}$ or when $S=\emptyset$ for which by Lemma \ref{H1G iso to H^2F} $H^1_\et(\CS,\un{G}) \cong H^2_\et(\CS,F(\un{G}))$
is trivial.

\begin{cor}
Let $\un{G} = \un{\textbf{SL}}_1(A)$ where $A$ is a quaternion $\CS$-algebra. 
\begin{itemize} 
\item [(1)] 
If $\Nrd: A^\times \to \CS^\times$ is not surjective, then the Hasse principle does not hold for $\un{G}$. 
\item [(2)] 
If the generic fiber $G$ is isotropic at $S$, then $\TwG$ is in bijection as a pointed-set with the abelian group $\Pic(\CS)/2 \times {_2}\Br(\CS)$. 
\end{itemize} 
\end{cor}


\begin{proof} 
$(1)$ The generic fiber $\textbf{SL}_1(A)$ is simply-connected thus due to Harder $H^1(K,\textbf{SL}_1(A))=1$, 
which indicates that $\un{\textbf{SL}}_1(A)$ admits a single genus (cf. Section \ref{torsors classification}),
i.e.,  $H^1_\et(\CS,\un{\textbf{SL}}_1(A))$ is equal to $\ClS(\un{\textbf{SL}}_1(A))$.  
By the exactness of sequence \eqref{LES SL(A)}, $H^1_\et(\CS,\un{\textbf{SL}}_1(A))$ cannot vanish if $\Nrd(A^\times) \neq \CS^\times$. \\ 
(2) Being of type $\text{A}_1$, $\un{G} = \un{\textbf{SL}}_1(A)$ does not admit a non-trivial outer form,
which implies that $\TwG = H^1_\et(\CS,\un{G}^\ad)$.
The short exact sequence of the universal covering of $\un{G}^\ad = \un{\textbf{PGL}}_1(A)$ with fundamental group $\un{\mu}_2$,
induces the long exact sequence (cf. \eqref{delta}):
$$ H^1_\et(\CS,\un{\textbf{SL}}_1(A)) \to  H^1_\et(\CS,\un{\textbf{PGL}}_1(A)) \xrightarrowdbl{\dl_{\un{G}^\ad}} H^2_\et(\CS,\un{\mu}_2)  $$
in which since $H^1_\et(\CS,\un{\textbf{SL}}_1(A))$ is trivial (due to strong approximation when $G$ is isotropic at $S$), 
the rightmost term is isomorphic by Lemma \ref{decomposition of H2 F admissible} to $\Pic(\CS)/2 \times {_2}\Br(\CS)$. 
\end{proof}

\begin{example}
Let $C$ be the projective line defined over $\BF_3$ and $S = \{t,t^{-1}\}$.
Then $K=\BF_3(t)$ and $\CS = \BF_3[t,t^{-1}]$.
For the quaternion $\CS$-algebra $A=(i^2=-1,j^2=-t)_{\CS}$ we get:
$$ \forall x,y,z,w \in \CS: \ \Nrd(x + yi + zj + wk) = x^2 + y^2 + t(z^2+w^2) $$
which shows that $\Nrd(A^\times) = \CS^\times = \BF_3^\times \cdot t^n, n \in \Z$. 
As $\CS$ is a UFD, the Hasse principle holds for $\un{G} = \un{\textbf{SL}}_1(A)$,   
though its generic fiber $G \cong \mathrm{Spin}_q$ for $q(x,y,z)=x^2+y^2+tz^2$ is anisotropic at $S$ (cf. \cite[Lemma~6]{IKR}).  
Also as $|{_2}\Br(\CS)|=2^{|S|-1}=2$ we have two distinct classes in $\TwG$, namely, $[\un{G}]$ and $[\un{G}^\op]$. 
\end{example}

Similarly, applying \'etale cohomology to the exact sequence of smooth $\CS$-groups:
$$ 1 \to \un{\textbf{SU}}(B,\s) \to \un{\textbf{U}}(B,\s) \xrightarrow{\Nrd} \un{\textbf{GL}}_1(R) \to 1 $$
induces the exactness of:
$$ 1 \to R^\times / \Nrd(\un{\textbf{U}}(B,\s)(\CS)) \to H^1_\et(\CS,\un{\textbf{SU}}(B,\s)) \to H^1_\et(\CS,\un{\textbf{U}}(B,\s)) \xrightarrow{\Nrd_*} H^1_\et(\CS,\textbf{Aut}(R)). $$

Let $A = D(B,\s)$ be the discriminant algebra.
If $R$ splits, namely, $R \cong \CS \times \CS$, then $B \cong A \times A^\text{op}$ and $\s$ is the exchange involution.
Then $\un{\textbf{U}}(B,\s) \cong \un{\textbf{GL}}_1(A)$ and $\un{\textbf{SU}}(B,\s) \cong \un{\textbf{SL}}_1(A)$,
so we are back in the previous situation.

\begin{cor}
If $\un{\textbf{U}}(B,\s)(\CS) \xrightarrow{\Nrd} R^\times$ is not surjective then the Hasse-principle does not hold for $\un{\textbf{SU}}(B,\s)$. 
\end{cor}

\mk

\section{In the Zariski topology} \label{Zariski}
A $\un{G}$-torsor $P$ is \emph{Zariski}, 
if the twisted form ${^P}\un{G}$ is generically and locally everywhere away of $S$ isomorphic to $\un{G}$, 
i.e., if it belongs to the principal genus of $\un{G}$ (see Section \ref{torsors classification}).
Let $\un{G}_0$ be a quasi-split semisimple $\CS$-group with an almost-simple generic fiber.
The continuous morphism between the categories of open subsets of $\CS$:
$(\CS)_\et \to (\CS)_\zar$ results -- given a variety $X$ defined over $\CS$ --
in the opposite inclusion of cohomology sets $H^r_\zar(\CS,X) \subseteq H^r_\et(\CS,X)$ for all $r>0$.
The restriction of the decomposition \eqref{Gille0}
\begin{equation}
\textbf{Twist}(\un{G_0}) \cong H^1_\et(\CS,\mathrm{Aut}(\un{G}_0)) \cong
\coprod_{[P]} H^1_\et(\CS,{^P}(\un{G}_0^\ad))/ \Theta(\CS)
\end{equation}
to Zariski torsors gives (compare with \cite[p.181]{Har}):
\begin{align} \label{Zariski twisted forms}
\textbf{Twist}_\zar(\un{G}_0) \cong H^1_\zar(\CS,\mathrm{Aut}(\un{G}_0)) \cong H^1_\zar(\CS,\un{G}^\ad_0)/\Theta(\CS).   
\end{align}
But as aforementioned, $H^1_\zar(\CS,\un{G}^\ad_0)$ is equal to the principal genus of $\un{G}_0^\ad$
on which the action of $\Theta(\CS)$ is trivial, hence \eqref{Zariski twisted forms} refines to:
\begin{equation} \label{Zariski twisted forms no action}
\textbf{Twist}_\zar(\un{G}_0) \cong H^1_\zar(\CS,\un{G}_0^\ad).
\end{equation}
Moreover, restricting the bijection $H^1_\et(\CS,\un{G}_0^\ad) \cong H^2_\et(\CS,F(\un{G}_0^\ad))$ (Lemma \ref{H1G iso to H^2F})
to the Zariski topology, $H^1_\zar(\CS,\un{G}_0^\ad)$ can be replaced with $H^2_\zar(\CS,F(\un{G}_0^\ad))$.
All twisted forms of $\un{G}_0$ in the Zariski topology being $K$-isomorphic are isotropic, so this time this includes groups of type~$\text{A}$.
Suppose $F(\un{G}_0^\ad)$ is admissible, splitting over an \'etale extension $R$ of $\CS$.
Then $H^1_\zar(\CS,\un{\BG}_m) \cong \Pic(\CS)$ while as $R$ is locally factorial $H^2_\zar(R,\un{\BG}_m)$ is trivial (\cite[Remark 3.5.1]{CTS})
thus $i(F(\un{G}_0))$ as well (see Definition \ref{i}).
Hence similarly as was done for $H^2_\et(\CS,F(\un{G}_0^\ad))$ in Lemma \ref{decomposition of H2 F admissible},
we get that $H^2_\zar(\CS,F(\un{G}_0^\ad)) \cong j(F(\un{G}_0))$.

\begin{cor}
Let $\un{G}_0$ be a semisimple $\CS$-group with an almost-simple generic fiber and an admissible fundamental group.
Then: $\textbf{Twist}_\zar(\un{G}_0) \cong j(F(\un{G}_0))$.
\end{cor}

\mk

{\bf Acknowledgements:}
The authors thank P.~Gille, B.~Kunyavski\u\i\ and A.~Qu\'eguiner-Mathieu for valuable discussions concerning the topics of the present article.
They would like also to thank the anonymous referee for a careful reading and many constructive remarks.

\mk

\end{document}